\documentclass[12pt]{amsart}
\usepackage{amssymb,amsmath,amsthm,comment}
\usepackage{pdfpages,graphicx} 
\usepackage[section]{placeins} 
\usepackage{wrapfig}
\usepackage{float}

\usepackage{setspace}
\newcommand{\shrinkmargins}[1]{
  \addtolength{\textheight}{#1\topmargin}
  \addtolength{\textheight}{#1\topmargin}
  \addtolength{\textwidth}{#1\oddsidemargin}
  \addtolength{\textwidth}{#1\evensidemargin}
  \addtolength{\topmargin}{-#1\topmargin}
  \addtolength{\oddsidemargin}{-#1\oddsidemargin}
  \addtolength{\evensidemargin}{-#1\evensidemargin}
  }
\shrinkmargins{0.5}

\newtheorem{theorem}{Theorem}
\newtheorem{lemma}[theorem]{Lemma}

\newtheorem{corollary}[theorem]{Corollary}

\newtheorem{conjecture}{Conjecture}
\newtheorem{proposition}[theorem]{Proposition}
{Claim}

\newtheorem{assumption}{Assumption}

\theoremstyle{remark}

\newtheorem*{remark}{Remark}
\newtheorem*{remarks}{{\bf Remarks}}

\numberwithin{theorem}{section} \numberwithin{equation}{section}

\usepackage{caption}
\usepackage{multirow}

\begin{document}
\title[]{Polynomization of the \\Chern--Fu--Tang conjecture}
\author{Bernhard Heim }
\address{Faculty of Mathematics, Computer Science, and Natural Sciences,
RWTH Aachen University, 52056 Aachen, Germany}
\email{bernhard.heim@rwth-aachen.de}
\author{Markus Neuhauser}
\address{Kutaisi International University (KIU), 5/7, Youth Avenue,  Kutaisi, 4600, Georgia}
\email{markus.neuhauser@kiu.edu.ge}
\subjclass[2010] {Primary 05A17, 11P82; Secondary 05A20}
\keywords{Integer Partitions, Polynomials, Partition Inequality}


\begin{abstract}
Bessenrodt and Ono's work on additive and multiplicative properties of the partition function and
DeSalvo and Pak's paper on the log-concavity 
of the partition function have generated many beautiful theorems and conjectures. 
In January 2020, the first author gave a lecture at the MPIM in Bonn on
a conjecture of Chern--Fu--Tang, and presented an extension (joint work with Neuhauser) involving polynomials.
Partial results have been announced. Bringmann, Kane, Rolen and Tripp provided complete proof of the Chern--Fu--Tang conjecture,
following advice from Ono to utilize a recently provided exact formula for the fractional partition functions.
They also proved a large proportion of Heim--Neuhauser's conjecture, which 
is the polynomization of Chern--Fu--Tang's conjecture. 
We prove several cases, not covered by Bringmann et.\ al. Finally, we lay out a general approach 
for proving the conjecture.
\end{abstract}

\maketitle
\newpage
\section{Introduction and main results}
Chern, Fu and Tang \cite{CFT18} conjectured an inequality for $k$-colored partition functions.
A partition of $n$ is called $k$-colored if each part can appear in $k$ colors and the number of these 
partitions has been denoted by $p_{-k}(n)$.
\begin{conjecture} [Chern, Fu, Tang 2018] \label{CFT} Let $n > m \geq 1$ and $k \geq 2$, except for $(k,n,m) = (2,6,4)$, then
\begin{equation}
p_{-k} (n-1) \, p_{-k} (m+1)  \geq \, p_{-k} (n) \, p_{-k} (m).
\end{equation}
\end{conjecture}
The conjecture has been motivated by two results. The first was
the work of Nicolas \cite{Ni78} and DeSalvo and Pak \cite{DP15} on the log-concavity of the 
partition function $p(n)= p_{-1}(n)$, $n >25$.
The second was the work of
Bessenrodt and Ono \cite{BO16}, and Alanazi, Gagola and Munagi \cite{AGM17} on an inequality involving
additive and multiplicative properties of the partition function. 
The conjecture is based on numerical evidence (\cite{CFT18}, Table 1).
For $b=a-2$, the conjecture implies the log-concavity for $p_{-k}(n)$ with respect to $n$ for $n \geq 3$ , $ k \geq 2$. One has to exclude the
case $k=2$ and $n=5$, since $ \left(p_{-2}(5)\right)^2 < p_{-2}(4) \, p_{-2}(6)$.

In \cite{HN19B} we proposed a polynomization of the Bessenrodt--Ono inequality. We also refer to 
recent work by Beckwith and Bessenrodt \cite{BB16}, Dawsey and Masri \cite{DM19}, and Hou and Jagadeesan \cite{HJ18}.
We transferred the inequality of the discrete $k$-colored partition
function to an inequality between
values of polynomials $P_n(x)$, defined as the coefficients of the $q$-expansion of all
powers of the Dedekind $\eta$-function \cite{On03}:
\begin{equation}
\sum_{n=0}^{\infty} P_n(z) \, q^n = \prod_{n=1}^{\infty} 
\left( 1 - q^n \right)^{-z}, \qquad ( q,z \in \mathbb{C}, \, \vert q \vert < 1).
\end{equation}
The polynomials can easily be recorded, for example $P_0(x)=1, P_{1}(x)=x, P_2(x)=(x+3)x/2$. They have interesting properties.
The $k$-colored partition function $p_{-k}(n)$ is equal to $P_n(k)$. Further, let for example, 
the root $x=-3$ of $P_2(x)$ be given. Then among all $2$nd coefficients of all non-trivial complex powers of $\prod_n (1-q^n)$, 
the coefficient assigned  to the $3$rd power vanishes.

Let $a,b \in \mathbb{N}$ with $a+b >2$ and $x \in \mathbb{R}$ with $x>2$. Then the inequality states:
\begin{equation}\label{BO}
P_a(x) \, \cdot \, P_b(x) > P_{a+b}(x).
\end{equation}
The proof was provided in \cite{HNT20}.

Building on 
Chern--Fu--Tang's result for $k=2$ and the positivity of the derivative of
$P_{a,b}(x):=
P_a(x) \, \cdot \, P_b(x)  - P_{a+b}(x)$ for $x >2$, we proposed an extension of
the Chern--Fu--Tang conjecture \cite{He20}.
\begin{conjecture}[Heim, Neuhauser] \label{HN: Conjecture 2}
Let $a > b \geq 0$ be integers. Then for all $x \geq 2$:
\begin{equation}\label{co}
\Delta_{a,b}(x) := P_{a-1}(x) P_{b+1}(x) - P_{a}(x) P_{b}(x) \geq 0,
\end{equation}
except for $b=0$ and $(a,b) = (6,4)$. The inequality (\ref{co}) is still true for $x \geq 3$ for $b=0$ and
for $x \geq x_{6,4}$ for $(a,b)=(6,4)$. Here $x_{a,b}$ is the largest real root of $\Delta_{a,b}(x)$.
\end{conjecture}
\begin{remarks}\hfill
\begin{itemize}
\item[a)]Conjecture 2 implies Conjecture 1.
\item[b)]We have $\Delta_{a,b}(0)=0$ and $\Delta_{a,a-1}(x)= 0$.
The leading coefficient of the polynomial $\Delta_{a,b}(x)$ is equal
to $\frac{a-b-1}{a! \, (b+1)!}$
for $a > b+1$. Thus, we have
\begin{equation*}
\lim_{x \rightarrow \infty} \Delta_{a,b}(x) = \infty. 
\end{equation*}
\item[c)] We have $\Delta_{a,0}(2)> 0$ and $\Delta_{a,1}(2)> 0$ for $a > 4$. This follows from \cite{HNT20}.
\item[d)] The case $b=0$ follows from (\ref{BO}) using properties of $P_{a-1,1}(x)$ \cite{HNT20}.
\item[e)] In \cite{He20} the case $b=1$ was already announced (proof is given in this paper).
\item[f)] The conjecture as stated in \cite{He20} for $(a,b)=(6,4)$ is refined. Note that $\Delta_{6,4}(2)<0$, 
which does not allow $\Delta_{6,4}(2)\geq 0$ for all $x >2$. 
This was also observed during the presentation (see also \cite{BKRT20}, remark related to Conjecture 2).
\end{itemize}
\end{remarks}
Expanding on an exact formula for the fractional partition function (in terms of Kloosterman sums and Bessel functions)
by Iskander, Jain and Talvola \cite{IJT20}, recently, Bringmann, Kane, Rolen, and Tripp \cite{BKRT20}
proved that for all positive real numbers $x_1,x_2,x_3,x_4$ and $n_1,n_2,n_3,n_4 \in \mathbb{N}$:
\begin{equation}
P_{n_{1}}(x_{1}) \, P_{n_{2}}(x_{2}) \geq P_{n_{3}}(x_{3}) \, P_{n_{4}}(x_{4}), 
\end{equation}
with respect to some general assumptions. They also obtained an explicit version.
We recall their result. Let $f(x) = O_{\leq}\left(g(x)\right)$ mean
that $\vert f(x) \vert \leq g(x)$ in some domain.
\begin{theorem}[Bringmann, Kane, Rolen, Tripp 2020] \label{BKRT}
Fix $ x \in \mathbb{R}$ with $x \geq 2$, and let $a, b \in \mathbb{N}_{\geq 2}$ with $a > b+1$. 
Set $A:= a-1 - \frac{x}{24}$ and $B:= b - \frac{x}{24}$, we suppose
$B \geq \max \, \left\{ 2 \, x^{11}, \frac{100}{x -24} \right\} $. Then 
\begin{eqnarray*}
\Delta_{a,b}(x) & = & P_{a-1}(x) \,  P_{b+1}(x) -  P_{a}(x) \,  P_{b}(x)
\\
& = &
\pi \, 
\left( \frac{x}{24}\right)^{\frac{x}{2} + 1}
(A B)^{-\frac{x}{4} - \frac{5}{4}} \, 
e^{\pi \sqrt{\frac{2 x}{3}} \left( \sqrt{A} + \sqrt{B} \right) }
\left( \sqrt{A} - \sqrt{B} \right) 
\left( 1 + O_{\leq}\left(\frac{2}{3} \right)\right).
\end{eqnarray*}
\end{theorem}
This leads to proof of the Chern--Fu--Tang conjecture and
to a large proportion of the Heim--Neuhauser conjecture.
We provide more details in the final section of this paper.
\begin{corollary} [Bringmann, Kane, Rolen, Tripp 2020] \label{cor: 1.2}
For any real number $x \geq 2$ and positive integers 
\begin{equation}
b \geq  B_0:= \max\, \left\{ 2 \, x^{11} + \frac{x}{24}, \frac{100}{x - 24} + \frac{x}{24} \right\}
\end{equation}
Conjecture 2 is true.
\end{corollary}
\begin{corollary} [Bringmann, Kane, Rolen, Tripp 2020]
The conjecture of Chern--Fu--Tang (Conjecture 1) is true. In particular $p_{-2}(n)$ is log-concave for $n \geq 6$, and
$p_{-k}(n)$ is log-concave for all $n$ and $k \in  \mathbb{N}_{\geq 3}$.
\end{corollary}
In this paper we show that Conjecture $1$ and Conjecture $2$ are closely related to
the Bessenrodt--Ono inequality: $x \, P_{a-1}(x) \geq P_{a}(x)$.  The appearing rational function
$\frac{P_{b+1}(x)}{P_b(x)}$ will be approximated by a linear factor, depending on $b$. 

We prove Conjecture $2$ for $b \in \{0,1,2,3\}$, and all integers $a >b$ and all real numbers $x \geq x_0=2$.
Further, in the odd cases $b=1$ and $b=3$, Conjecture $2$ is already true for $x \geq 1$. 
To prove that $\Delta_{a,b}(x) \geq 0$, we study $\Delta_{a,b}(x_0) \geq 0$ and prove that $\Delta_{a,b}'(x)>0 $ for all $x > x_0$.
We believe that this approach is the most direct method to prove Conjecture 2.

The positivity of the derivative is expected, since $\Delta_{a,b}(x) >0$ for $x \geq x_0$ is
a statement on the largest real root $x_{a,b}$ of  $\Delta_{a,b}(x)$ and the observed property, that the real part of
the complex root seems to be smaller than $x_{a,b}$ (see Figure \ref{f1}).

\begin{minipage}{1.0\textwidth}
\begin{center}
\includegraphics[width=0.52\textwidth]{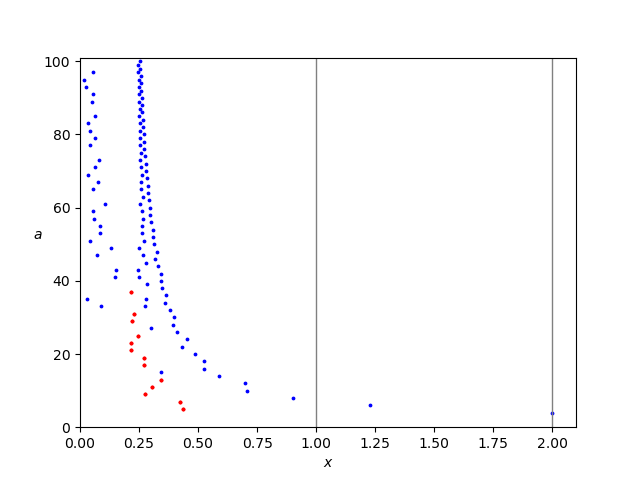}
\captionsetup{margin={0cm,0cm,0cm,0cm}}
\captionof{figure}{Roots of $\Delta_{a,2}(x)$ with a positive real part. Blue = real root, red = complex root.} \label{f1}
\end{center}
\end{minipage}

As already mentioned, the case $b=0$ has been almost proved in \cite{HNT20}. The complete statement and proof is given in
Section 2.
In Section 3 we prove the following results.

\begin{theorem} \label{main}
Let $ a \in \mathbb{N}$, $b \in \{1,2,3\}$ and $x \in \mathbb{R}$.
For $b$ odd we put $x_0:=1$ and for $b$ even we put $x_0:=2$.
Let $a_0= a_0(b):= b+2$. Then
\begin{equation}
\Delta_{a,b}(x) = P_{a-1}(x) \, P_{b+1}(x) - P_a(x) P_b(x) > 0.
\end{equation}
for all $a \geq a_0$ and $x > x_0$.
\end{theorem}

The cases $x=x_{0}$ will be stated in Proposition ~\ref{prop}
and Corollary \ref{b}.
There the strict inequality does not hold in general. 
It fails for example for $(a,b,x_0) =( 4,2,2)$. Further, we obtain:
\begin{corollary}
\label{ableitungdeltapositiv}Let
$b \in \{1,2,3\}$ be given. Then $\Delta_{a,b}'(x) >0$ for all $a \geq a_0$ and
$x > x_0$.
\end{corollary}
In Section $4$ we outline a program to attack all cases of Conjecture 2. Finally, in Section $5$, we
provide some numerical data. All computations have been done using PARI/GP or Julia.

\section{The polynomials $\Delta_{a,b}(x)$ and $P_{a,b}(x)$}
There is a direct connection between the polynomialized Chern--Fu--Tang inequality (\ref{co}) and the
Bessenrodt--Ono inequality (\ref{BO}).
Let $a \geq 1 $, then
\begin{equation}
\Delta_{a,0}(x) = P_{a-1}(x) P_1(x) - P_a(x) P_{0}(x) = P_{a-1,1}(x).
\end{equation}
In the following we assume that $a \geq 2$, since $\Delta_{1,0}(x)$ = 0.
It is obvious that $\Delta_{a,0}(0) =0$, $\Delta_{a,0}'(0) <0$, and
that $\lim_{x \rightarrow \infty} \Delta_{a,0}(x) = \infty$.
Let us record the first polynomials and several properties. Let $Z_n$ be the set of roots of $\Delta_{a,0}$ and 
$x_{a,0}$ be the largest real root.
\begin{table}[H]
\begin{equation*}
\begin{array}{r@{\, \cdot \,}lll}
2 & \Delta_{2,0}(x) = x (x-3)       &           Z_2 = \{ 0, 3\}    & x_{2,0} = 3 \\
3 & \Delta_{3,0}(x) = x (x^2 -4x)   &       Z_3 = \{ -2,0,2 \}     & x_{3,0} = 2 \\
8 & \Delta_{4,0}(x) = x \left( x^3 + 6 x^2 + 9x -14\right)  &  Z_4 = \{ -7,-1,0,2 \} & x_{4,0} = 2 \\
30 & \Delta_{5,0}(x) = x \left( x^4 + 15 x^3 + 20 x^2 -60x -36 \right) &  
Z_5 = \{ \ldots , 0, x_{5,0}\} & x_{5,0} \approx 1.69 \\
\end{array}
\end{equation*}
\caption{Polynomials $\Delta _{a,0}\left( x\right) $, their sets of roots, and  largest real roots.}
\end{table}
\begin{theorem}[\cite{HNT20}]
Let $a >2$. Then for all $x>2$ we have the property
\begin{equation}
\Delta_{a,0}(x) = P_{a-1,1}(x) >0.
\end{equation}
Let $a=2$. Then $\Delta_{2,0}(3)=0$ and for all $x >3$ we have the strict inequality $\Delta_{2,0}(x) >0$.
We further have $\Delta_{3,0}(2) = \Delta_{4,0}(2)=0$. Let $a >4$ and $x \geq 2$, then  we have
$\Delta_{a,0}(x)>0$.
\end{theorem}
We deduce from (\cite{HNT20}, proof of Proposition 5.1) the following result.
\begin{corollary}
Let $a >2$ and $x \geq 2$. Then $\Delta_{a,0}'(x)>0$. 
\end{corollary}
Further, we have:
\begin{lemma}
Let $a \geq 5$. Then there exists a real number $\alpha $, $1 < \alpha <2$, such that $\Delta_{a,0}(\alpha)=0$.
Let $x_{a,0}$ be the largest real root of $\Delta_{a,0}(x)$. Then $ 1 < x_{a,0} < 2$ and $ \Delta_{a,0}(x) >0$ for all $x > x_{a,0}$.
\end{lemma}
For $b \in \{ 0,1,2 \}$ we have the following useful property.
\begin{proposition}\label{prop}
Let $x_0 =2$ and let $b \in \{ 0,1,2 \}$. Then $\Delta_{a,b}(x_0)>0$ for $a \geq 5$. Let $b=1$, then
this is already true for $a \geq 3$. The bounds for $b=0$ and $b=2$ are sharp.
\end{proposition}
\begin{proof}
The following quotients are all larger or equal to $x_0$.
Let $b \in \{0,1,2\}$. Then $\frac{P_{b+1}(x_0)}{P_{b}(x_0)} \geq x_0$:
$$ \frac{P_1(x_0)}{P_0(x_0)} = x_0, \quad 
\frac{P_2(x_0)}{P_1(x_0)} = \frac{5}{2} > x_0, \,\,   \text{ and } \,\,  \frac{P_3(x_0)}{P_2(x_0)} = 2 = x_0. $$
Thus, $\Delta_{a,b}(x_0) \geq P_{b}(x_0) \Delta_{a,0}(x_0)$ and $\Delta_{a,b}(x_0) > 0$ for $a \geq 5$.
The explicit shape and values of the involved polynomials for $a \leq 4$
complete the proof:
\begin{eqnarray*}
\Delta_{3,1}(x) & = & \frac{ x^2}{12} \left( x^2 +11 \right) \\
\Delta_{4,1} (x)& = & \frac{x^2}{24} \left( x^3 + 6x^2 + 11 x + 6 \right).
\end{eqnarray*}
We have $\Delta_{3,0}(x_0) = \Delta_{4,0}(x_0) = 0$ and $\Delta_{4,2}(x_0)<0$.
\end{proof}

\section{Log-concavity of partition numbers}
Nicolas \cite{Ni78} and DeSalvo and Pak \cite{DP15} proved the log-concavity of the
partition function $p(n)$ for $n \geq 26$:
\begin{equation}\label{lc}
p(n)^2 - p(n-1) p(n+1) \geq 0.
\end{equation}

Note that (\ref{lc})
fails for all $1 \leq n \leq 25$ odd, but is still true for $n$ even.
Explicit study of the small cases (Table \ref{firstcases}) leads to the following refined result:
\begin{proposition}\label{exceptions}
Let $q(n):= p(n)/p(n-1)$. 
Then $q(n+2) \leq q(n)$ for all $n \in \mathbb{N}$ and $q(27)  \geq q(n)$ for all $n \geq 27$.
For $n \leq 27$ we have the following chain:
\begin{eqnarray*}
q(2) &>& q(4) > q(6) > q(3) > q(8) > q(5)=q(10) > q(12) > q(7)=q(9) 
\\ &>&q(14) 
>q(11) > q(16) > q(13) >q(15) > q(18) > q(17) > q(20) \\ &>& q(19) >q(22) > q(21) > q(24) > q(23) > q(26) > q(25) > q(27).
\end{eqnarray*}
\end{proposition}

\begin{table}[H]
\[
\begin{array}{|r|c||r|c|}
\hline
n & q\left( n\right) \approx & n & q\left( n \right) \approx \\ \hline \hline
1 & 1.00000000 & 16 & 1.31250000 \\ \hline
2 & 2.00000000 & 17 & 1.28571429 \\ \hline
3 & 1.50000000 & 18 & 1.29629630 \\ \hline
4 & 1.66666667 & 19 & 1.27272727 \\ \hline
5 & 1.40000000 & 20 & 1.27959184 \\ \hline
6 & 1.57142857 & 21 & 1.26315789 \\ \hline
7 & 1.36363636 & 22 & 1.26515152 \\ \hline
8 & 1.46666667 & 23 & 1.25249501 \\ \hline
9 & 1.36363636 & 24 & 1.25498008 \\ \hline
10 & 1.40000000 & 25 & 1.24317460 \\ \hline
11 & 1.33333333 & 26 & 1.24412666 \\ \hline
12 & 1.37500000 & 27 & 1.23563218 \\ \hline
13 & 1.31168831 & 28 & 1.23521595 \\ \hline
14 & 1.33663366 & 29 & 1.22781065 \\ \hline
15 & 1.30370370 & 30 & 1.22760131 \\ \hline
\end{array}
\]
\caption{Approximate values of $q\left( n \right) $ for $1\leq n \leq 30$.} \label{firstcases}
\end{table}
\begin{corollary}\label{b}
Let $a$ and $b$ be positive integers. Let $ a > b+1$. Then 
\begin{equation} \label{pos}
\Delta_{a,b}(1) \geq 0
\end{equation}
is true for all $b$ odd and for all $b \geq 26$.
For $ 1 < b \leq 26$ even 
we have the following
result. Inequality (\ref{pos}) is satisfied for
$a\in A_{0}\left( b\right) \cup \left\{ a\in \mathbb{N}:a\geq a_{1}\left( b\right) \right\} $
from the following Table \ref{T2}.
\begin{table}[H]
\[
\begin{array}{|c||c|c|c|c|c|c|c|c|c|c|c|c|c|}
\hline
b & 2 & 4 & 6 & 8 & 10 & 12 & 14 & 16 & 18 & 20 & 22 & 24 & 26 \\ \hline \hline
A_{0}\left( b\right) & \left\{ 5\right\} & \left\{ 7\right\} & \left\{ 9,11\right\} & \left\{ 11\right\} & \left\{ 13\right\} & \left\{ 15\right\} & \emptyset & \emptyset & \emptyset & \emptyset & \emptyset & \emptyset & \emptyset \\ \hline
a_{1}\left( b\right) & 7 & 9 & 13 & 13 & 15 & 17 & 17 & 19 & 21 & 23 & 25 & 27 & 28 \\ \hline
\end{array}
\]
\caption{Data when inequality (\ref{pos}) is satisfied.} \label{T2}
\end{table}
\end{corollary}
\begin{proof}
The proof follows from Proposition \ref{exceptions} and 
\begin{equation}
\Delta_{a,b}(1) \geq 0 \Longleftrightarrow q(a) \leq q(b+1).
\end{equation}
\end{proof}

\section{Proof of Theorem \ref{main}}
Let us first recall a formula \cite{Ko04} for the coefficients of $P_n(x)$. Let $P_n(x)
= \sum_{m=1}^n A_{n,m} \, x^m$. Then 
\begin{equation}
A_{n,m} = \frac{1}{m!}
\sum_{\substack {k_1, \ldots, k_m \in \mathbb{N} \\ k_1+ \ldots + k_m=n}} 
\, \prod_{i=1}^{m} 
\frac{\sigma (k_i)}{k_i}.
\label{eq:kostant}
\end{equation}
\subsection{Case $b=1$ and Theorem \ref{main} for $x_0=1$.}
We prove here that $\Delta_{a,1}(x)>0$ for all $a \geq 3$ and
$x >x_0=1$. 

\begin{proof}
Corollary \ref{b} implies $\Delta_{a,1}(x_0) \geq 0$ for all $a \geq 3$.
Note that $\Delta_{a,1}(x)$ has degree $a+1$ and has non-negative coefficients for $ 2 < a < 6$.
This implies that the theorem is already true for $x >0$.
We have $\Delta _{6,1}\left( x\right) >0$ for
$x\geq x_{0}$.
Let
$F_{a}\left( x\right) =P_{a-1}\left( x\right) \frac{x+3}{2} -P_{a}\left( x\right) $.
Then $\Delta _{a,1}\left( x\right) =xF_{a}\left( x\right) $.
Therefore to
show that $\Delta _{a,1}\left( x\right) > 0$ it is sufficient to show
that $F_{a}\left( x\right) >0$.

This we prove  by induction on $a\geq 3$ for $x >x_0$.
Note that
$F_{a}\left( x\right) >0$ for $x>x_{0}$ and
$3\leq a \leq 6$.
Therefore in the induction step we  assume $ a \geq 7$ and that
$F_{m}
\left( x\right) >0$ is true for all
$3 \leq m < a$ and $x > x_0$.
Now we will show $F_{a}^{\prime }\left( x\right) >0 $
for all $x > x_0$.
The derivative $F_{a}^{\prime }\left( x\right) $ is equal to
\begin{eqnarray*}
&&P_{a-1}'(x) \, \frac{x+3}{2} + P_{a-1}(x) \,  \frac{1}{2}
- P_{a}'(x) \\
&=&\sum_{k=1}^{a-1} \frac{\sigma(k)}{k} P_{a-1-k}(x) \, \frac{x+3}{2} - 
\sum_{k=1}^{a} \frac{\sigma(k)}{k} P_{a-k}(x)  + P_{a-1}(x) \,  \frac{1}{2}.
\end{eqnarray*}
This follows from \cite{HN18}:
\begin{equation} \label{derivative}
P_n'(x) = \sum_{k=1}^n \frac{\sigma(k)}{k} \, P_{n-k}(x).
\end{equation}
By the induction hypothesis we obtain
\begin{equation*}
\sum_{k=1}^{a-1} \frac{\sigma(k)}{k} P_{a-1-k}(x) \, \frac{x+3}{2} >
\frac{\sigma(a-1)}{a-1} \frac{x+3}{2} + 
\sum_{k=1}^{a-2} \frac{\sigma(k)}{k} P_{a-k}(x)
\end{equation*}
and
\begin{eqnarray*}
F_{a}^{\prime }\left( x\right)  & > &
\frac{\sigma(a-1)}{a-1} \frac{x+3}{2} - \sum_{k=a-1}^{a} \frac{\sigma(k)}{k} P_{a-k}(x)  + P_{a-1}(x) \,  \frac{1}{2}   \\
& = &
\frac{1}{2} \, P_{a-1}(x) - \frac{\sigma(a-1)}{a-1}  \frac{x}{2} +  \frac{3\sigma(a-1)}{2\, (a-1)} - 
\frac{  \sigma(a)}{a}.
\end{eqnarray*}
In the last step we utilize the property $ a  < \sigma(a) < a \left( 1 + \ln(a) \right)$ and obtain
\begin{equation*}
F_{a}^{\prime }\left( x\right)  > \frac{1}{2} \, P_{a-1}(x) - \frac{\sigma \left( a-1\right) x}{2\left( a-1\right) } 
+\frac{3}{2}-\left( 1 + \ln(a) \right) .
\end{equation*}
Since $P_{a-1}\left( 1\right) $ is the partition number of $a-1$ we have
$P_{a-1}\left( 1\right) \geq a-1$ for $x_0=1$.
The coefficients of the
polynomial $P_{a-1}\left( x\right) $ are provided by (\ref{eq:kostant})
and it implies that the coefficients of
$P_{a-1}\left( x\right) -\frac{\sigma \left( a-1\right) }{a-1}x$ are
not negative.
Therefore we obtain
$P_{a-1}\left( x\right) -\frac{\sigma \left( a-1\right) }{a-1}x\geq P_{a-1}\left( 1\right) -\frac{\sigma \left( a-1\right) }{a-1}$
for $x\geq 1$.
Finally
$F_{a}^{\prime }\left( x\right) >\frac{a-1}{2}-\frac{1+
\ln \left( a-1\right) }{2}+\frac{3}{2}-1-\ln \left( a\right) >0$
for $a\geq 7$.
\end{proof}

\begin{proof}[Proof of Corollary \ref{ableitungdeltapositiv} for the case $b=1$]
We have shown in the previous proof that $F_{a}\left( x\right) >0$ and
$F_{a}^{\prime }\left( x\right) >0$ for $x>x_{0}$ and $a\geq 7$. Therefore also
$\Delta _{a,1}^{\prime }\left( x\right) =xF_{a}^{\prime }\left( x\right) +F_{a}\left( x\right) >0$
for $x>x_{0}$. We also mentioned in the previous proof that the coefficients of
$\Delta _{a,1}\left( x\right) $ are not negative for $2<a<6$. For $a=6$ it can
be checked directly that $\Delta _{6,1}^{\prime }\left( x\right) >0$ for
$x\geq x_{0}$. This proves Corollary~\ref{ableitungdeltapositiv} for $b=1$.
\end{proof}

\subsection{Case $b=2$ and Theorem \ref{main} for $x_0=2$.}
\begin{proof}
Let $x_0=2$.
We have $P_{2}\left( x\right) =\left( x+3\right) \frac{x}{2}$
and
$P_{3}\left( x\right) =\left( x+8\right) \left( x+1\right) \frac{x}{6}$. Let
$F_{a}\left( x\right) =\frac{x+4}{3}P_{a-1}\left( x\right) -P_{a}\left( x\right) $.
Since
$\left( x+8\right) \left( x+1\right) 
 \geq \left( x+4\right) \left( x+3\right) $
for $x\geq 2$
we obtain
$$\Delta _{a,2}\left( x\right)
\geq \left( x+3\right) 
\frac{x}{2}\left( \frac{x+4}{3}P_{a-1}\left( x\right) -P_{a}\left( x\right) 
\right) =\left( x+3\right) \frac{x}{2}F_{a}\left( x\right). $$
For $x=x_{0}$ we have equality.

We will show $F_{a}\left( x\right) > 0$ by induction
on $a$. It holds for $a=4$ as
$F_{4}\left( x\right)
=\left( x+1\right) \left( x-1\right) \left( x-2\right) \frac{x}{72}>0$
for $x>2$. Similarly, we can show that $F_{a}\left( x\right) >0$ for
$x>2$ for $5\leq a\leq 13$. The
following
proposition shows that  $F_{a}^{\prime }\left( x\right) >0$ for $x>x_{0}$
if we assume $F_{m}\left( x\right) >0$ for $x>x_{0}$ for $4\leq m<a$.

The last step in the induction
is the following.
In Proposition \ref{prop} we showed that
$\Delta _{a,2}\left( x_{0}\right) \geq 0$ for $a\geq 4$. Additionally,
$F_{a}\left( x_{0}\right) =\frac{\Delta _{a,2}\left( x_{0}\right) }{\left( x_{0}+3\right) \frac{x_{0}}{2}}\geq 0$.
Using
$F_{a}^{\prime }\left( x\right) >0$ for $x>x_{0}$ we can conclude that
$\Delta _{a,2}\left( x\right) \geq \left( x+3\right) \frac{x}{2}F_{a}\left( x\right) >0$
for $x>x_{0}$.
\end{proof}

\begin{proposition}
Let $a\geq 14$ and assume that
$F_{m}\left( x\right) =\frac{x+4}{3}P_{m-1}\left( x\right) -P_{m}\left( x\right) >0$
for $x>x_{0}=2$
and $4\leq m<a$. Then $F_{a}^{\prime }\left( x\right) >0$.
\end{proposition}

\begin{proof}
For $F_{a}^{\prime }\left( x\right)$ we obtain
\begin{equation*}
\frac{1}{3}P_{a-1}\left( x\right) +\frac{x+4}{3}\sum _{k=1}^{a-1}\frac{\sigma \left( k\right) }{k}P_{a-1-k}\left( x\right) -\sum _{k=1}^{a}\frac{\sigma \left( k\right) }{k}P_{a-k}\left( x\right).
\end{equation*}
We apply the assumptions and obtain
\begin{eqnarray*}
F_{a}^{\prime }\left( x\right)
&>&\frac{1}{3}P_{a-1}\left( x\right) +\frac{x+4}{3}\sum _{k=a-3}^{a-1}\frac{\sigma \left( k\right) }{k}P_{a-1-k}\left( x\right) -\sum _{k=a-3}^{a}\frac{\sigma \left( k\right) }{k}P_{a-k}\left( x\right) \\
&= &\frac{1}{3}P_{a-1}\left( x\right) +\frac{x+4}{3}\left( \frac{\sigma \left( a-3\right) }{a-3}\left( x+3\right) \frac{x}{2}+\frac{\sigma \left( a-2\right) }{a-2}x+\frac{\sigma \left( a-1\right) }{a-1}\right) \\
&&{}-\frac{\sigma \left( a-3\right) }{a-3}\left( x+8\right) \left( x+1\right) \frac{x}{6}-\frac{\sigma \left( a-2\right) }{a-2}\left( x+3\right) \frac{x}{2}-\frac{\sigma \left( a-1\right) }{a-1}x-\frac{\sigma \left( a\right) }{a}\\
&=&\frac{1}{3}P_{a-1}\left( x\right) \\
&&{}-\frac{\sigma \left( a-3\right) }{a-3}\left( x-2\right) \frac{x}{3}-\frac{\sigma \left( a-2\right) }{a-2}\left( x+1\right) \frac{x}{6}-\frac{\sigma \left( a-1\right) }{a-1}\frac{2x-4}{3}-\frac{\sigma \left( a\right) }{a}\\
&\geq &\frac{1}{3}P_{a-1}\left( x\right)
-\frac{1+\ln \left( a\right) }{6}\left( x+1\right) \left( 3x-2\right) .\\
\end{eqnarray*}
We now use that
\begin{eqnarray*}
P_{a-1}\left( x\right)
&=&
\frac{\sigma \left( a-1\right) }{a-1}x+\sum _{k=1}^{a-2}\frac{\sigma \left( a-1-k\right) \sigma \left( k\right) }{2\left( a-1-k\right) k}x^2+\sum _{m=3}^{a-1}A_{a-1,m}x^{m}\\
&\geq &
x+\frac{a-2}{2}x^{2}+\sum _{m=3}^{a-1}A_{a-1,m}x^{m}
\end{eqnarray*}
for $x>0$. Therefore
\begin{eqnarray*}
F_{a}^{\prime }\left( x\right)
&>&\frac{x}{3}+\frac{a-2}{6}x^{2}+\frac{1}{3}\sum _{m=3}^{a-1}A_{a-1,m}x^{m}-\frac{1+\ln \left( a\right) }{6}\left( x+1\right) \left( 3x-2\right) \\
&=&\frac{1+\ln \left( a\right) }{3}+\frac{1-\ln \left( a\right) }{6}x+\frac{a-5-3\ln \left( a\right) }{6}x^{2}+\frac{1}{3}\sum _{m=3}^{a-1}A_{a-1,m}x^{m}\\
&=&\frac{4a-16-12\ln \left( a\right) }{6}+\frac{4a-19-13\ln \left( a\right) }{6}\left( x-2\right) \\
&&{}+\frac{a-5-3\ln \left( a\right) }{6}\left( x-2\right) ^{2}+\frac{1}{3}\sum _{m=3}^{a-1}A_{a-1,m}x^{m}>0
\end{eqnarray*}
for $a\geq 14$ and $x>x_{0}=2$.
\end{proof}

\begin{proof}[Proof of Corollary \ref{ableitungdeltapositiv} for the case $b=2$]
From the previous proof of Theorem \ref{main} for the case
$b=2$ we see that $F_{a}\left( x\right) >0$ for $x>x_{0}=2$. The previous
proposition showed that $F_{a}^{\prime }\left( x\right) >0$ for $a\geq 14$.
Therefore
$\Delta _{a,2}^{\prime }\left( x\right) =\left( x+\frac{3}{2}\right) F_{a}\left( x\right) +\left( x+3\right) \frac{x}{2}F_{a}^{\prime }\left( x\right) >0$
for $a\geq 14$. The remaining cases for $4\leq a\leq 13$ can be checked
directly.
\end{proof}

\subsection{Case $b=3$ and Theorem \ref{main} for $x_{0}=1$.}
\begin{proof}
We have
$P_{3}\left( x\right) =\frac{x}{6}\left( x+1\right) \left( x+8\right) $
and
$P_{4}\left( x\right) =\frac{x}{24}\left( x+1\right) \left( x+3\right) \left( x+14\right) $.
As
$\left( x+3\right) \left( x+14\right) 
\geq \frac{1}{3}\left( x+8\right) \left( 3x+17\right) $
for $x\geq 1$ we obtain
$$P_{4}\left( x\right) \geq \frac{x}{72}\left( x+1\right) \left( x+8\right) \left( 3x+17\right). $$
Let
$F_{a}\left( x\right) =\frac{3x+17}{12}P_{a-1}\left( x\right) -P_{a}\left( x\right) $.
Then
\begin{equation}
\Delta _{a,3}\left( x\right) =P_{a-1}\left( x\right) P_{4}\left( x\right) -
P_{a}\left( x\right) P_{3}\left( x\right) \geq \frac{x}{6}\left( x+1\right) \left( x+8\right) F_{a}\left( x\right)
\label{eq:deltaa3}
\end{equation}
for $x\geq 1$. Note that for $x=x_{0}=1$ we have equality.
We also have
$F_{a}\left( x\right) >0$ for $x>1$ and $5\leq a\leq 14$.

The proof will be by induction on $a$. The following proposition shows that
$F_{a}^{\prime }\left( x\right) >0$ for $x>x_{0}$ and $a\geq 15$, if we assume
that $F_{m}\left( x\right) >0$ for $x>x_{0}$ and $5\leq m<a$.

By Corollary \ref{b} $\Delta _{a,3}\left( x_{0}\right) \geq 0$ for $a\geq 5$.
Additionally,
$F_{a}\left( x_{0}\right) =\frac{\Delta _{a,3}\left( x_{0}\right) }{\frac{x_{0}}{6}\left( x_{0}+1\right) \left( x_{0}+8\right) }\geq 0$.
Using $F_{a}^{\prime }\left( x\right) >0$ for $x>x_{0}$ and (\ref{eq:deltaa3})
we can conclude that
$\Delta _{a,3}\left( x\right) \geq \frac{x}{6}\left( x+1\right) \left( x+8\right) F_{a}\left( x\right) >0$.
\end{proof}

\begin{proposition}
Let $a\geq 15$. If
$F_{m}\left( x\right) =\frac{3x+17}{12}P_{m-1}\left( x\right) -P_{m}\left( x\right) >0$
for $x>x_{0}=1$ for all
$5\leq m<a$ then $F_{a}^{\prime }\left( x\right) >0$.
\end{proposition}

\begin{proof} The derivative $F_{a}^{\prime }\left( x\right)$ is equal to
\begin{equation*}
\frac{1}{4}P_{a-1}\left( x\right) +\frac{3x+17}{12}\sum _{k=1}^{a-1}\frac{\sigma \left( k\right) }{k}P_{a-1-k}\left( x\right) -\sum _{k=1}^{a}\frac{\sigma \left( k\right) }{k}P_{a-k}\left( x\right).
\end{equation*}
Applying the assumptions leads to 
\begin{eqnarray*} 
&&F_{a}^{\prime }\left( x\right) \\
&>&\frac{1}{4}P_{a-1}\left( x\right) +\frac{3x+17}{12}\sum _{k=a-4}^{a-1}\frac{\sigma \left( k\right) }{k}P_{a-1-k}\left( x\right) -\sum _{k=a-4}^{a}\frac{\sigma \left( k\right) }{k}P_{a-k}\left( x\right) \\
&=&\frac{1}{4}P_{a-1}\left( x\right) +\frac{3x+17}{12}
( \frac{\sigma \left( a-4\right) }{a-4}P_{3}\left( x\right) +\frac{\sigma \left( a-3\right) }{a-3}P_{2}\left( x\right) +\frac{\sigma \left( a-2\right) }{a-2}x\\
&&{}+\frac{\sigma \left( a-1\right) }{a-1}
) -\frac{\sigma \left( a-4\right) }{a-4}P_{4}\left( x\right) -\frac{\sigma \left( a-3\right) }{a-3}P_{3}\left( x\right) -\frac{\sigma \left( a-2\right) }{a-2}P_{2}\left( x\right) \\
&&{}-\frac{\sigma \left( a-1\right) }{a-1}x-\frac{\sigma \left( a\right) }{a} \\
&=&\frac{1}{4}P_{a-1}\left( x\right) -\frac{\sigma \left( a-4\right) }{a-4}\frac{5x}{36}\left( x+1\right) \left( x-1\right) -\frac{\sigma \left( a-3\right) }{a-3}\frac{x}{24}\left( x^{2}+10x-19\right) \\
&&{}-\frac{\sigma \left( a-2\right) }{a-2}\frac{x}{12}\left( 3x+1\right) -\frac{\sigma \left( a-1\right) }{a-1}\frac{9x-17}{12}-\frac{\sigma \left( a\right) }{a}.
\end{eqnarray*}
Now, $
x^{2}+10x-19\leq x^{2}+10x-11=\left( x+11\right) \left( x-1\right)
$
and $9x-17\leq 9x-9$. As $-x<0$ we obtain
\begin{eqnarray*}
&&F_{a}^{\prime }\left( x\right) \\
&>&\frac{1}{4}P_{a-1}\left( x\right) -\frac{\sigma \left( a-4\right) }{a-4}\frac{5x}{36}\left( x+1\right) \left( x-1\right) -\frac{\sigma \left( a-3\right) }{a-3}\frac{x}{24}\left( x-1\right) \left( x+11\right) \\
&&{}-\frac{\sigma \left( a-2\right) }{a-2}\frac{x}{12}\left( 3x+1\right) -\frac{\sigma \left( a-1\right) }{a-1}\frac{3x-3}{4}-\frac{\sigma \left( a\right) }{a}\\
&\geq &\frac{1}{4}P_{a-1}\left( x\right) -\frac{1+\ln \left( a\right) }{72}\left( 13x^{3}+48x^{2}+17x+18\right) .
\end{eqnarray*}
Now
\begin{eqnarray*}
&&P_{a-1}\left( x\right) \\
&=&\frac{\sigma \left( a-1\right) }{a-1}x+\sum _{k=1}^{a-2}\frac{\sigma \left( a-1-k\right) \sigma \left( k\right) }{2\left( a-1-k\right) k}x^{2}\\
&&{}+\sum _{j=1}^{a-3}\sum _{k=1}^{a-j-2}\frac{\sigma \left( a-1-j-k\right) \sigma \left( k\right) \sigma \left( j\right) }{6\left( a-1-j-k\right) jk}x^{3}+\sum _{m=4}^{a-1}A_{a-1,m}x^{m} \\
&\geq &
x+\frac{a-2}{2}x^{2}+\sum _{j=1}^{a-3}\frac{a-j-2}{6}x^{3}+\sum _{m=4}^{a-1}A_{a-1,m}x^{m} \\
&=&x+\frac{a-2}{2}x^{2}+\binom{a-2}{2}\frac{x^{3}}{6}+\sum _{m=4}^{a-1}A_{a-1,m}x^{m}.
\end{eqnarray*}
Then
\begin{eqnarray*}
F_{a}^{\prime }\left( x\right)
&>&\frac{1}{4}x+\frac{a-2}{8}x^{2}+\binom{a-2}{2}\frac{x^{3}}{24}+\frac{1}{4}\sum _{m=4}^{a-1}A_{a-1,m}x^{m}\\
&&{}-\frac{1+\ln \left( a\right) }{72}\left( 13x^{3}+48x^{2}+17x+18\right) \\
&=&-\frac{1+\ln \left( a\right) }{4}+\frac{1-17\ln \left( a\right) }{72}x+\frac{3a-22-16\ln \left( a\right) }{24}x^{2}\\
&&{}+\frac{3\left( a-2\right) \left( a-3\right) -26-26\ln \left( a\right) }{144}x^{3}+\frac{1}{4}\sum _{m=4}^{a-1}A_{a-1,m}x^{m}\\
&=&\frac{a^{2}+a-58-64\ln \left( a\right) }{48}+\frac{9a^{2}-9a-286-304\ln \left( a\right) }{144}\left( x-1\right) \\
&&{}+\frac{3a^{2}-9a-52-58\ln \left( a\right) }{48}\left( x-1\right) ^{2}\\
&&{}+\frac{3\left( a-2\right) \left( a-3\right) -26-26\ln \left( a\right) }{144}\left( x-1\right) ^{3}+\frac{1}{4}\sum _{m=4}^{a-1}A_{a-1,m}x^{m}
>0
\end{eqnarray*}
for $a\geq 15$ and $x>x_{0}=1$.
\end{proof}

\begin{proof}[Proof of Corollary \ref{ableitungdeltapositiv} for the case $b=3$]
From the proof of Theorem \ref{main} for the case $b=3$ we observe that
$F_{a}\left( x\right) >0$ for $x>x_{0}=1$. The previous proposition shows that
$F_{a}^{\prime }\left( x\right) >0$ for $a\geq 15$ and $x>x_{0}$. Therefore
$\Delta _{a,3}^{\prime }\left( x\right) =\frac{1}{6}\left( 3x^{2}+18x+8\right) F_{a}^{\prime }\left( x\right) +\frac{x}{6}\left( x+1\right) \left( x+8\right) F_{a}^{\prime }\left( x\right) >0$.
For the remaining cases $5\leq a\leq 14$ it can be checked directly that
$\Delta _{a,3}^{\prime }\left( x\right) >0$.
\end{proof}


\section{Conjecture 2: approach for general $b$.}
We offer a general approach to Conjecture 2, based on four assumptions.
Let $x_0>0$ and $a > b+1$. We define
\begin{eqnarray}
H_{b}(x) & := &
\frac{P_{b+1}(x)}{P_{b}(x)} - \frac{x}{b+1},\\
G_{b}(x) & := & \frac{x}{b+1}+H_{b}\left( x_{0}\right) =\frac{x-x_{0}}{b+1}+
\frac{P_{b+1}\left( x_{0}\right) }{P_{b}\left( x_{0}\right) }, \\
F_{a,b}(x) & := & G_{b}\left( x\right) P_{a-1}\left( x\right) -P_{a}\left( x\right) 
.\label{eq:fab}
\end{eqnarray}
\subsection{\label{four}Four Assumptions}

In this subsection let $a>b+1$ and $x_{0}>0$ be fixed.

\begin{assumption}
\label{annahme0}$\Delta _{a,b}\left( x_{0}\right) \geq 0$.
\end{assumption}

\begin{assumption} $H_b(x) \geq H_b(x_0)$
\label{annahme1}
for all $x\geq x_{0}$.
\end{assumption}

\begin{assumption} For all $x>x_{0}$
and $a-1-b\leq k\leq a-1$ let
\label{annahme2}
\begin{equation}
G_b(x)
P_{a-1-k}\left( x\right) -P_{a-k}\left( x\right) \leq 0.
\end{equation}
\end{assumption}

\begin{assumption}[Induction hypothesis]
\label{eq:induktionsannahme}$F_{m,b}\left( x\right) >0$
for $x>x_{0}$ and $b+2\leq m<a$.
\end{assumption}

\begin{remarks}\hfill
\begin{enumerate}
\item  Let $H_b(x)$ be 
monotonically increasing for $x\geq x_{0}$, then Assumption~\ref{annahme1} is valid.
\item  \label{deltaab0}Assumption~\ref{annahme1} implies
\begin{equation}
\Delta _{a,b}\left( x\right) \geq P_{b}\left( x\right) F_{a,b}\left( x\right) .
\label{eq:deltaab}
\end{equation}
For $x=x_0$ we have equality.
\end{enumerate}
\end{remarks}
The idea is to generalize the induction step approach on $a>b+1$ from the
previous section
to arbitrary $b$
and show as the main intermediate step
\begin{equation}
F_{a,b}^{\prime }\left( x\right) \geq 0.
\label{eq:induktionsschritt}
\end{equation}
Then from part \ref{deltaab0} of the previous remarks we obtain
$F_{a,b}\left( x_{0}\right) =\frac{\Delta _{a,b}\left( x_{0}\right) }{P_{b}\left( x_{0}\right) }$.
Assumption \ref{annahme0} implies
$F_{a,b}\left( x\right) \geq 0$ and together
with (\ref{eq:deltaab}) we obtain also
$\Delta _{a,b}\left( x\right) \geq 0$ for $x\geq x_{0}$.

Using the assumptions is not sufficient to complete the induction
step. The last estimate on $F_{a,b}\left( x\right) $ can in general
yet only be bounded
asymptotically for large $a$, see next subsection.

For now we are going to explain how we can use the assumptions from the
beginning of this subsection for a lower bound on $F_{a,b}\left( x\right) $.
If we derive (\ref{eq:fab}) we obtain
\begin{eqnarray*} F_{a,b}^{\prime }\left( x\right)
&=&\frac{1}{b+1}P_{a-1}\left( x\right) +
G_b(x)
P_{a-1}^{\prime }\left( x\right) -P_{a}^{\prime }\left( x\right) \\
&=&\frac{1}{b+1}P_{a-1}\left( x\right) +
G_b(x) 
\sum _{k=1}^{a-1}\frac{\sigma \left( k\right) }{k}P_{a-1-k}\left( x\right) -\sum _{k=1}^{a}\frac{\sigma \left( k\right) }{k}P_{a-k}\left( x\right).
\end{eqnarray*}
Using now Assumption \ref{eq:induktionsannahme} we obtain
\begin{eqnarray*}
F_{a,b}^{\prime }\left( x\right) 
&>&\frac{1}{b+1}P_{a-1}\left( x\right) +
G_b(x)
\sum _{k=a-1-b}^{a-1}\frac{\sigma \left( k\right) }{k}P_{a-1-k}\left( x\right)
{}-\sum _{k=a-1-b}^{a}\frac{\sigma \left( k\right) }{k}P_{a-k}\left( x\right) \\
&=&\frac{1}{b+1}P_{a-1}\left( x\right) +\sum _{k=a-1-b}^{a-1}\frac{\sigma \left( k\right) }{k}
\left( 
G_b(x)
P_{a-1-k}\left( x\right) -P_{a-k}\left( x\right) \right) -\frac{\sigma \left( a\right) }{a}
\end{eqnarray*}
and with Assumption \ref{annahme2} we can continue
\begin{equation}
\begin{array}{rcl}
F_{a,b}^{\prime }\left( x\right) 
&>&\displaystyle
\frac{1}{b+1}P_{a-1}\left( x\right) \\
&&\displaystyle {}
-\left( 1+\ln \left( a\right) \right) \left( 1+\sum _{k=a-1-b}^{a-1}\left( P_{a-k}\left( x\right) -
G_b(x)
P_{a-1-k}\left( x\right) \right) \right) .
\end{array}
\label{eq:vorlaguerre}
\end{equation}

\subsection{\label{laguerre}An estimate using associated Laguerre polynomials}

Here we will explain an idea how to show the positivity of
right hand side of (\ref{eq:vorlaguerre}).
We want to bound the coefficients of $P_{a-1}\left( x\right) $ from below
in such a way that they
dominate the coefficients of the subtracted polynomial. Unfortunately also this
approach here in the end
only works asymptotically and only for most coefficients.

Let
$L_{n}^{\left( 1\right) }\left( x\right) =\sum _{k=0}^{n}\binom{n+1}{n-k}\frac{\left( -x\right) ^{k}}{k!}$
be the
associated Laguerre polynomial of degree $n$ with parameter
$\alpha =1$. Then
$P_{n}\left( x\right) \geq \frac{x}{n}L_{n-1}^{\left( 1\right) }\left( -x\right) =\sum _{k=1}^{n}\binom{n-1}{k-1}\frac{x^{k}}{k!}$
for $x>0$. This follows from \cite{HN19A} or directly from Kostant's formula
(\ref{eq:kostant}) as it implies
$A_{n,m}\geq \frac{1}{m!}\sum _{\substack{k_{1},\ldots ,k_{m}\in \mathbb{N} \\ k_{1}+\ldots +k_{m}=n}}1=\frac{1}{m!}\binom{n-1}{m-1}$.
From the last step in the previous subsection
we can now continue
\begin{eqnarray*}F_{a,b}^{\prime }\left( x\right)
&>&\frac{1}{b+1}\sum _{k=1}^{a-1}\binom{a-2}{k-1}\frac{x^{k}}{k!}\\
&&{}-\left( 1+\ln \left( a\right) \right) \left( 1+\sum _{k=a-1-b}^{a-1}\left( P_{a-k}\left( x\right) -
G_b(x)
P_{a-1-k}\left( x\right) \right) \right) .
\end{eqnarray*}
This is positive
if we can bound the coefficients of the subtracted polynomial with the
$1+\ln \left( a\right) $ term by the coefficients
$\frac{1}{b+1}\binom{a-2}{k-1}$. This is always possible for $2\leq k\leq a-2$ for large
$a\geq a_{0}$.

Then for example for $x_{0}=2$ we could deduce
from \cite{BKRT20}
that $\Delta _{a,b}\left( x_{0}\right) \geq 0$. As explained
shortly before then we
also have
$F_{a,b}\left( x_{0}\right) =\frac{\Delta _{a,b}\left( x_{0}\right) }{P_{b}\left( x_{0}\right) }\geq 0$.
Therefore, $F_{a,b}\left( x\right) >0$ for $x>x_{0}$.
Then (\ref{eq:deltaab})
implies
$\Delta _{a,b}\left( x\right) \geq P_{b}\left( x\right) F_{a,b}\left( x\right) >0$
for $x>x_{0}$.

\subsection{Proof of Assumptions \ref{annahme1} and \ref{annahme2} 
for $b\in \left\{ 0,1,2,3,4,5,6\right\} $}

Our approach
to prove Assumption \ref{annahme2} requires us to compare in
particular the values for the initial points $x_{0}$ for $b$
with all values for $k<b$. As it turns out again the case
$b=5$ then forces us to choose $x_{0}\geq 2.0554$ in this
approach which then propagates to the case $b>5$.
For values of $b<5$ we could also have chosen $x_{0}=2$
for example, compare Table \ref{kleinste}.

Having proven Assumptions \ref{annahme1} und \ref{annahme2} for the
cases $b\in \left\{ 4,5,6\right\} $ carries out the
induction step up to inequality (\ref{eq:vorlaguerre}).
What is left to do is to prove that the right hand side
of (\ref{eq:vorlaguerre}) is really
positive (and to check that $\Delta _{a,b}\left( x_{0}\right) \geq 0 $
for $x_{0}=2.0554$). The positivity can probably be shown using the method
proposed in the last subsection.
So the analysis of bounding the coefficients of the subtracted
polynomial with the $1+\ln \left( a\right) $ (see the end of the
previous subsection) has to be carried out in
the cases $b\in \left\{ 4,5,6\right\} $.

\begin{proposition}
For $b\in \left\{ 0,1,2,3,4,5,6\right\} $ the functions
$x\mapsto \frac{P_{b+1}\left( x\right) }{P_{b}\left( x\right) }-\frac{x}{b+1}$
are monotonically increasing for $x\geq x_{0}\geq 0.776$
which implies
Assumption \ref{annahme1}.
\end{proposition}

\begin{remark}
Actually the proof will show that the functions are monotonically increasing for
all $x> 0$
for $b\in \left\{ 0,1,2,3,4,6\right\} $
with the exception of $b=5$ where we need the restriction
on $x_{0}$.
\end{remark}

\begin{proof}
The derivative is
\begin{equation}
\frac{P_{b+1}^{\prime }\left( x \right) P_{b}\left( x\right) -P_{b+1}\left( x\right) P_{b}^{\prime }\left( x\right) }{\left( P_{b}\left( x\right) \right) ^{2}}-\frac{1}{b+1}.
\label{eq:ableitung}
\end{equation}
Let
\begin{equation}
N_{b}\left( x\right) =P_{b+1}^{\prime }\left( x\right) P_{b}\left( x\right) -P_{b+1}\left( x\right) P_{b}^{\prime }\left( x\right)
.
\label{eq:nenner}
\end{equation}
Then (\ref{eq:ableitung}) is not negative if and only if
$N_{b}\left( x\right) -\frac{1}{b+1}\left( P_{b}\left( x \right) \right) ^{2}\geq 0$.
Now
\begin{table}[H]
\[
\begin{array}{|r|l|}
\hline
b & N_{b}\left( x\right) -\frac{1}{b+1}\left( P_{b}\left( x\right) \right) ^{2} \\ \hline \hline
0 & 0
 \\ \hline
1 & 0
 \\ \hline
2 & \frac{5}{6}\*x^2
 \\ \hline
3 & \frac{5}{24}\*\left( x+1\right) ^{2}\*x^2
 \\ \hline
4 & \frac{1}{48}\*\left( x^{2}+4x+16\right) \*\left( x+3\right) ^{2}\*x^{2}
 \\ \hline
5 & \frac{1}{4320}\*\left( 5\*x^{6}
 + 120\*x^{5}
 + 1250\*x^{4}
 + 6144\*x^{3}
 + 11705\*x^{2}
 - 1800\*x
 - 9000\right) \*x^{2}
 \\ \hline
6 & \frac{1}{120960}\*( 5\*x^8
 + 220\*x^7
 + 4090\*x^6
 + 38416\*x^5
 + 192565\*x^4
 + 536500\*x^3
 + 1049420\*x^2 \\
& {}+ 1440000\*x
 + 763008) \*x^{2}
 \\ \hline
\end{array}
\]
\caption{Polynomials $N_{b}\left( x\right) -\frac{1}{b+1}\left( P_{b}\left( x\right) \right) ^{2}$ for $b\in \left\{ 0,1,2,3,4,5,6\right\} $.}
\end{table}
which are all not negative for $x\geq x_{0}$.
\end{proof}

\begin{proposition}
Let $b\in \left\{ 1,2,3,4,5,6\right\} $ then
$x\mapsto \frac{P_{k+1}\left( x\right) }{P_{k}\left( x\right) }-\frac{x}{b+1}$
is monotonically increasing for $x>x_{0}=2.0554$ and $0\leq k\leq b$ and
\begin{equation}
\frac{P_{b+1}\left( x_{0}\right) }{P_{b}\left( x_{0}\right) }-\frac{x_{0}}{b+1}\leq \frac{P_{k+1}\left( x_{0}\right) }{P_{k}\left( x_{0}\right) }-\frac{x_{0}}{b+1}
\label{eq:anfang2}
\end{equation}
for $0\leq k\leq b$. This implies Assumption \ref{annahme2}
for $x_{0}=2.0554$.
\end{proposition}

\begin{proof}
Deriving the functions
$x\mapsto \frac{P_{k+1}\left( x\right) }{P_{k}\left( x\right) }-\frac{x}{b+1}$
for $0\leq k\leq b$ we obtain  using similarly  $N_{k}\left( x\right) $
from (\ref{eq:nenner})
$\frac{N_{k}\left( x\right) -\frac{1}{b+1}\left( P_{k}\left( x\right) \right) ^{2}}{\left( P_{k}\left( x\right) \right) ^{2} }$.
This is not negative if and only if the numerator is not negative. Obviously
this is larger than
$
N_{k}\left( x\right) -\frac{1}{k+1}\left( P_{k}\left( x\right) \right) ^{2}$
which we have seen to be not negative in the proof of the previous proposition.
What remains to check is that
$\frac{P_{b+1}\left( x_{0}\right) }{P_{b}\left( x_{0}\right) }
\leq \frac{P_{k+1}\left( x_{0}\right) }{P_{k}\left( x_{0}\right) }$
for $0\leq k\leq b-1$, compare the following table.
\begin{table}[H]
\[
\begin{array}{|c||c|c|c|c|c|c|}
\hline
b & 1 & 2 & 3 & 4 & 5 & 6 \\ \hline
\frac{P_{b+1}\left( x_{0}\right) }{P_{b}\left( x_{0}\right) }\approx & 2.527700 & 2.025772 & 2.017982 & 1.819048 & 1.819044 & 1.707376 \\ \hline
\end{array}
\]
\caption{\label{aaaa}Approximate values of $\frac{P_{b+1}\left( x_{0}\right) }{P_{b}\left( x_{0}\right) }$ for $b\in \left\{ 1,2,3,4,5,6\right\} $.}
\end{table}
\end{proof}

For fixed $b$
we can also determine the smallest $x_{0}$ for which (\ref{eq:anfang2}) holds.
\begin{table}[H]
\[
\begin{array}{|r|l|}
\hline
b & x_{0}\approx \\ \hline \hline
2 & 2\\ \hline
3 & 2 \\ \hline
4 & 1.6881868943126478278636511038164231908 \\ \hline
5 & 2.0553621798507231766687152242721716951 \\ \hline
6 & 1.5657320643972915718958748689518846691 \\ \hline
\end{array}
\]
\caption{\label{kleinste}Approximate smallest $x_{0}$ for which (\ref{eq:anfang2}) holds.}
\end{table}

\subsection{Partial result}
Unfortunately we cannot show Assumption~\ref{annahme1} yet, but we can show the following weaker version.

\begin{lemma}
If we assume that $\Delta _{b+1,B}\left( x\right) >0$ for
all $x>x_{0}>0$ and $0\leq B\leq b-1$ then
$\frac{P_{b+1}\left( x\right) }{P_{b}\left( x\right) }$
is monotonically increasing for $x>x_{0}$.
\end{lemma}

\begin{proof}
If we consider its derivative we obtain
$\frac{P_{b+1}^{\prime }\left( x\right) P_{b}\left( x\right) -P_{b+1}\left( x\right) P_{b}^{\prime }\left( x\right) }{\left( P_{b}\left( x\right) \right) ^{2}}$.
The numerator is
\begin{eqnarray*}
&&\sum _{k=1}^{b+1}\frac{\sigma \left( k\right) }{k}P_{b+1-k}\left( x\right) P_{b}\left( x\right) -P_{b+1}\left( x\right) \sum _{k=1}^{b}\frac{\sigma \left( k\right) }{k}P_{b-k}\left( x\right) \\
&=& \frac{\sigma \left( b+1\right) }{b+1}P_{b}\left( x\right) +\sum _{k=1}^{b}\frac{\sigma \left( k\right) }{k}\left( P_{b+1-k}\left( x\right) P_{b}\left( x\right) -P_{b+1}\left( x\right) P_{b-k}\left( x\right) \right) .
\end{eqnarray*}
Now for $A=b+1$ and $B=b-k\leq b-1=A-2$ we can apply the
assumption and obtain that all
$P_{b+1-k}\left( x\right) P_{b}\left( x\right) -P_{b+1}\left( x\right) P_{b-k}\left( x\right) =\Delta _{A,B}\left( x\right) >0$
for $x>x_{0}$.
\end{proof}
%
%
%
\section{Concluding remarks}
We consider sequences $\{a_n\}_{n=0}^{\infty}$ with non-negative elements.
A sequence is log-concave if $a_n^2 - a_{n-1} \, a_{n+1} \geq 0$ for all $n \in \mathbb{N}$, and
strongly
log-concave if the inequalities are strictly positive. Let $c_:= \sum_{i+j =n} a_i \, b_j$ be
the convolution of two sequences. Hoggar \cite{Ho74} proved that the convolution of
two finite positive
(strongly) log-concave sequences is again (strongly) log-concave. 
Let $x_1$ and $x_2$ be complex numbers, then the convolution of the two sequences
$\{P_n(x_1)\}$  and $\{P_n(x_{2})\}$ is equal to $\{P_n(x_1 + x_2)\}$. Note that $P_n(x)>0$ for $x>0$. 

The link between $\Delta_{a,b} \geq 0$ and log-concavity is given by the following observation.
Let $x>0$ and always $a,b \in \mathbb{N}$ with $a> b+1$:
\begin{equation}
\Delta_{a,b}(x) \geq 0  \Longleftrightarrow \frac{P_{b+1}(x)}{P_b(x)} \geq \frac{P_{a}(x)}{P_{a-1}(x)}. 
\end{equation}
\begin{remarks} Let $x>0$ be given. \newline
a) Let $\Delta_{b+2,b}(x) \geq 0$ for all $ b \in \mathbb{N}_0$. Then $\{P_n(x)\}$ is log-concave. 
\newline
b) Let $\{P_n(x)\}$ be log-concave, then $\Delta_{a,b}(x) \geq 0$ for all $a >b+1$ and $b \in \mathbb{N}_0$.
\end{remarks}

Bringmann, Kane, Rolen and Tripp \cite{BKRT20} proved (see also Introduction), 
that there exists a constant $B_0 = B_0(x) := 
\max \left\{ 2x^{11}+\frac{x}{24},\frac{100}{x-24}+\frac{x}{24}\right\}$ for $x \geq 2$ such that
$\Delta_{a,b}(x) \geq 0$ for all $b \geq B_0(x)$ and $a \geq b+1$ (Table \ref{B}).
\begin{table}[H]
\[
\begin{array}{|r|r|}
\hline
x & \max \left\{ 2x^{11}+\frac{x}{24},\frac{100}{x-24}+\frac{x}{24}\right\} \approx \\ \hline \hline
2 & 4096.08333333 \\ \hline
3 & 354294.12500000 \\ \hline
4 & 8388608.16666667 \\ \hline
5 & 97656250.20833333 \\ \hline
\end{array}
\]
\caption{Approximate values of $\max \left\{ 2x^{11}+\frac{x}{24},\frac{100}{x-24}+\frac{x}{24}\right\} $.}
\label{B}
\end{table}

Let $x=k \in \mathbb{N}_{\geq 2}$, then 
$B_0(k) =2k^{11}+\frac{k}{24}$. Thus, $\Delta_{a,b}(x) \geq 0$ for fixed $x>0$ and all pairs $(a,b)$ with
$a \geq b+1$ and $b \in \mathbb{N}_{0}$ is equivalent to $\{P_n(x)\}$ log-concave. Now, by \cite{BKRT20}
it is sufficient to show that the quotients $\frac{P_{n}(x)}{P_{n-1}(x)}$ are decreasing when $n$ is increasing for
all $ 1 \leq n \leq B_0(x)$. 
In \cite{BKRT20} this last step had been executed successfully for $k=2$ and $n \geq 6$ and $k=3$ and all $n$.
The authors also invented some sophisticated computer calculations for $k=4$ and $k=5$.
Although, they still needed a $5$ day and a $71$ day long computer calculation for these cases.
Finally they proved that $\{P_n(2)\}$ is log-concave for $n \geq 6$ and $\{P_n(k)\}$ is log-concave
for $k=3,4$ and $5$. Applying the result of Hoggar finally proves the Chern--Fu--Tang conjecture.
Note that $\lim_{x \rightarrow \infty} B_0(x) = \infty$, which makes this method difficult to prove
Conjecture 2, for general $x\geq 2$. For $0< x <3$, the sequence $\{P_n(x)\}$ is never log-concave
(for small $n$) since $\Delta_{2,0}(x)<0$, which causes technical problems (see also $k=1$ and $k=2$, where
finitely many {\it exceptions} appear).
In this paper we offer an approach which takes care of $x \geq x_0$ bounded from below.
We fix $b$ and determine  $a_0 \in \mathbb{N}$ and $x_0 \in \mathbb{R}_{>0}$ such that
$\Delta_{a,b}\left( x\right) \geq 0$ for all $a \geq a_0$
and $x \geq x_0$. This takes into account that
exceptions may exist and allows for example to vary $a_0$ and $x_0$.
Let $b \in \{0,1,2,3\}$. We have determined $a_0$ and $x_0$ dependent on $b$, such that 
$\Delta_{a,b}(x) \geq 0$ and $\Delta_{a,b}'(x) \geq 0$ for $a \geq a_0$ and $x \geq x_0$.
\begin{minipage}{1.05\textwidth}
\begin{center}
\includegraphics[width=0.4\textwidth]{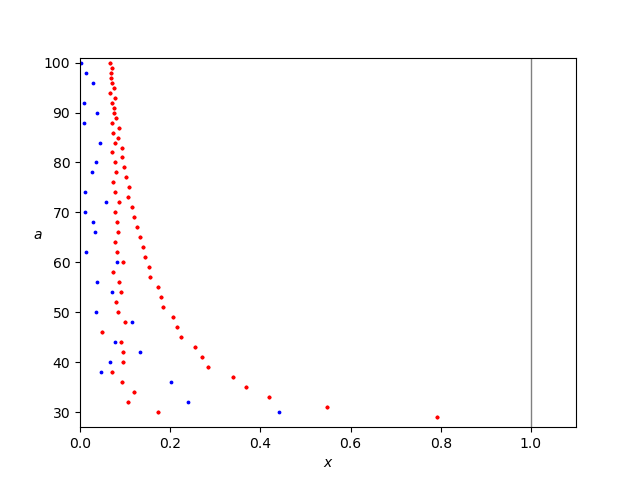}
\includegraphics[width=0.4\textwidth]{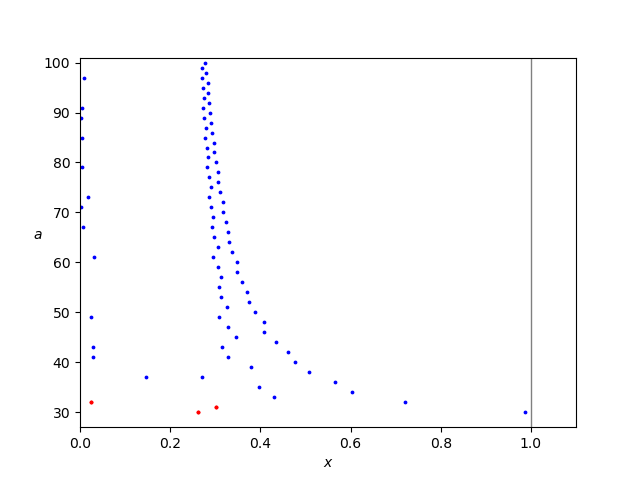}
\captionsetup{margin={0cm,0cm,0cm,0cm}}
\captionof{figure}{Roots of $\Delta_{a,27}(x)$ and $\Delta_{a,28}(x)$ with a positive real part. \\
Blue = real root, red = complex root.} \label{f2728}
\end{center}
\end{minipage}
\ \newline
We expect for $b \geq 27$ and $x_0=1$ that $a_0$ can be chosen as $b+2$. This can be considered as the
generic case (see Figure \ref{f2728} which illustrates this expectation). If we assume $0 < x_0 < 1$, then 
it is an interesting but challenging task to
determine $a_0 = a_0(b,x_0)$.
\newline
\newline
{\bf Acknowledgments.} To be entered later.

\end{document}